\documentclass[11pt]{article}

\usepackage{graphicx}
\usepackage{amsmath,amssymb,amsthm,hyperref}
\usepackage{fullpage}
\usepackage{enumerate}
\usepackage{enumitem}
\usepackage{pgf,tikz}
\usetikzlibrary{patterns}
\usetikzlibrary{arrows}
\usetikzlibrary{positioning}
\usepackage{thmtools,thm-restate}

\def\COMMENT#1{}
\let\COMMENT=\footnote
\newtheorem{theorem}{Theorem}[section]

\newtheorem{lemma}[theorem]{Lemma}

\newtheorem{claim}[theorem]{Claim}

\newtheorem{question}[theorem]{Question}
\newtheorem{corollary}[theorem]{Corollary}

\newtheorem{remark}[theorem]{Remark}

\theoremstyle{definition}
\newtheorem{define}[theorem]{Definition}

\newcommand{\qedclaim}{\hfill$\blacksquare$}
\newenvironment{proofclaim}{\removelastskip\penalty55\medskip\noindent{\it Proof of the claim. }}{\medskip}

\def\marrow{{\marginpar[\hfill$\longrightarrow$]{$\longleftarrow$}}}

\def\geza#1{{\sc G\'eza says: }{\marrow\sf #1}}

\title{Matchings avoiding ordered patterns}
\author{János Barát\thanks{Supported by ERC Advanced Grants “GeoScape”, no. 882971 and “ERMiD”, no. 101054936.\\
E-mail:\protect\href{mailto:freschi.andrea@renyi.hu}{barat@renyi.hu}
}\\
\small Alfréd R\'enyi Institute of Mathematics, Budapest, Hungary, and\\
\small University of Pannonia, Department of Mathematics, Veszprém, Hungary\\
\\
Andrea Freschi\thanks{
Supported by ERC Advanced Grants “GeoScape”, no. 882971 and “ERMiD”, no. 101054936.\\
E-mail:\protect\href{mailto:freschi.andrea@renyi.hu}{freschi.andrea@renyi.hu}.}\\
\small Alfréd R\'enyi Institute of Mathematics, Budapest, Hungary,\\ 
\\
Géza Tóth\thanks{
Supported by National Research, Development and Innovation Office, NKFIH,
K-131529, ADVANCED-152590 and ERC Advanced Grant “GeoScape" 882971.
E-mail:\protect\href{mailto:freschi.andrea@renyi.hu}{geza@renyi.hu}}
\\
\small Alfréd R\'enyi Institute of Mathematics, Budapest, Hungary}

\begin{document}

\maketitle

\begin{abstract}
    A vertex-ordered graph is a graph equipped with a linear ordering of its vertices.
    A pair of independent edges in an ordered graph can exhibit one of the following three patterns: separated, nested or crossing. 
    We say a pair of independent edges is non-separated if it is either crossing or nested.
    Non-nested and non-crossing pairs are defined analogously.
    
    We are interested in the following Tur\'an-type problems: for each of the aforementioned six patterns, determine the maximum number of edges of an $n$-vertex ordered graph that does not contain a $k$-matching such that every pair of edges exhibit the fixed pattern.
    Exact answers have already been obtained for four of the six cases.
    The main objective of this paper is to investigate the two remaining open cases, namely non-separated and non-nested matchings.
    
    We determine the exact maximum number of edges of an $n$-vertex ordered graph that does not contain a non-separated $k$-matching, which has the form $\frac{3}{2}(k-1)n+\Theta(k^2)$. 
    For the non-nested case, we show the maximum number of edges lies between $(k-1)n$ and $(k-1)n+\binom{k-1}{2}$. 
    We also determine the exact maximum number of edges of an $n$-vertex ordered graph that does not contain an alternating path of given length.
    We discuss some related problems and raise several 
    conjectures.
    Furthermore, our results and conjectures yield consequences to certain Ramsey-type problems for non-nested matchings and alternating paths. 

%
\end{abstract}

\section{Introduction}

A {\it (vertex-)ordered} graph is a graph equipped with a linear ordering on its vertex set.
For simplicity, we always assume that an $n$-vertex ordered graph has vertex set $[n]:=\{1,2,\dots,n\}$ i.e., the first~$n$ positive integers.

In recent years, there has been accelerating interest on Tur\'an-type questions for ordered graphs, see e.g.~\cite{tardos}.
For a family of unordered graphs~$\mathcal H$, the {\it Tur\'an number} $ex(n,{\mathcal H})$ denotes the maximum number of edges of an $n$-vertex graph, which does not contain a copy of any $H\in {\mathcal H}$.
The celebrated Erd\H os--Stone--Simonovits theorem determines~$ex({n,\mathcal H})$ asymptotically for $\mathcal H$ fixed, provided~$\mathcal H$ contains no bipartite graph.
Analogous notions exist for ordered graphs.
We say an ordered graph $G$ contains {\it a copy of an ordered graph $H$} if there exists an injective graph isomorphism from $H$ to $G$, which preserves the order of the vertices.

\begin{define}
Given a family of ordered graphs $\mathcal H$, we write $ex_<(n,\mathcal H)$ to denote the maximum number of edges in an ordered graph with vertex set~$[n]$ not containing any copy of $H\in\mathcal H$.
\end{define}

Pach and Tardos established an analogue of the Erd\H os--Stone--Simonovits theorem in the ordered setting.
The crucial parameter here is the {\it interval chromatic number $\chi_<(H)$}, defined as the size of a smallest proper interval colouring of~$H$.

\begin{theorem}[Pach and Tardos \cite{PT06}]\label{thm:PachTardos}
For any ordered graph~$H$, we have
    $$ex_<(n,H)=\left(1-\frac{1}{\chi_<(H)-1}+o(1)\right)\frac{n^2}{2}.$$
\end{theorem}

As in the unordered setting, Theorem~\ref{thm:PachTardos} does not provide an asymptotic formula when $\chi_<(H)=2$, or $|V(H)|$ and $n$ are comparable in size.
There has been great interest in studying the $\chi_<(H)=2$ case, sometimes assuming further that the underlying ordered graph can be partitioned into two independent intervals. 
This setting naturally corresponds to $(0,1)$-matrices.
The extremal theory of $(0,1)$-matrices include fundamental results by Füredi-Hajnal~\cite{fh}
and Marcus--Tardos solving the Stanley-Wilf conjecture~\cite{mt}.

In this paper, we are interested in Tur\'an-type problems for ordered matchings.
In order to formulate these questions precisely, we introduce some notation.
Let $\mathbb N$ denote the set of positive integers.
For every~$a,b\in\mathbb N$ with $a\le b$, we write $[a,b]:=\{a,a+1,\dots,b\}$ for the set of integers between~$a$ and~$b$ (included).
We call such a set an {\it interval} and the {\it size} of an interval is the number of integers in it. 

\begin{define}[Separated, nested and crossing edges]
    Let~$xy$ and~$uv$ be two independent (i.e., vertex-disjoint) edges with $x,y,u,v\in\mathbb N$ such that~$x<y$ and~$u<v$.
    We say~$xy$ and~$uv$ are
    \begin{itemize}
        \item {\it separated} if the intervals $[x,y]$ and $[u,v]$ are disjoint,
        \item {\it nested} if the interval $[x,y]$ contains $[u,v]$ or vice versa, and
       \item {\it crossing} otherwise.
    \end{itemize}
    See Figure~\ref{fig:patterns} for a drawing of these patterns.
\end{define}

\begin{figure}[h!]
\centering
\begin{tikzpicture}
\draw[ultra thick] (6,0) arc (180:0:0.6);
\draw[ultra thick] (6.6,0) arc (180:0:0.6);

\draw[ultra thick] (3,0) + (0,0) arc (180:0:0.9);
\draw[ultra thick] (3,0) + (0.6,0) arc (180:0:0.3);

\draw[ultra thick] (0,0) + (0,0) arc (180:0:0.3);
\draw[ultra thick] (0,0) + (1.2,0) arc (180:0:0.3);

\end{tikzpicture}
\caption{From left to right: separated, nested and crossing pairs.}
\label{fig:patterns}
\end{figure}
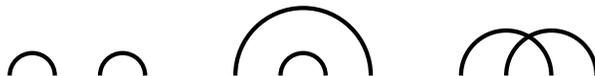

\begin{define}[$k$-matchings with special patterns]
    A {\it matching} is a collection of independent edges.
    A matching consisting of $k\in\mathbb N$ edges is called a {\it $k$-matching}.
    For any matching whose vertices are positive integers, we say such matching is:
    \begin{itemize}[itemsep=0em]
        \item {\it separated} if every pair of its edges are separated,
        \item {\it nested} if every pair of its edges are nested,
        \item {\it crossing} if every pair of its edges are crossing,
        \item {\it non-separated} if every pair of its edges are not separated,
        \item {\it non-nested} if every pair of its edges are not nested,
        \item {\it non-crossing} if every pair of its edges are not crossing.
    \end{itemize}
    As a convention, we denote a separated (resp. nested and crossing) $k$-matching as $\text{Sep-}\mathcal M_k$ (resp. $\text{Nest-}\mathcal M_k$ and $\text{Cross-}\mathcal M_k$).
    Similarly, we denote a non-separated (resp. non-nested and non-crossing) $k$-matching as $\neg\text{Sep-}\mathcal M_k$ (resp. $\neg\text{Nest-}\mathcal M_k$ and $\neg\text{Cross-}\mathcal M_k$).
\end{define}

Tur\'an-type problems for these special matchings have been previously considered in the literature, although not systematically, for example under a different terminology or setting.
The exact values of $ex_<(n,\,\cdot\,)$ for separated, nested, crossing and non-crossing $k$-matchings have been determined for all values of~$k$ and~$n$.
See the next subsection for a formal presentation of these results.

The main goal of this paper is to study the remaining two cases: non-separated and non-nested matchings.
A summary of all exact results, or best bounds, can be found in Table~\ref{table:1}.
This includes our new results for non-separated and non-nested matchings.

\begin{table}[h]
\centering
\bgroup
\def\arraystretch{1.75}
\begin{tabular}{|c|c|c|}
\hline
$k$-matching  & Maximum number of edges in~$[n]$ without a $k$-matching & References  \\
\hline
\hline
Separated  & 
    $|E(T(n+1,k))|-k+1$ & \cite{layout}\\
Nested  & $2(k-1)n - (k-1)(2k - 1)$ & \cite{layout}\\
Crossing  & $2(k-1)n-(k-1)(2k - 1)$ & \cite{CapoyleasP}\\
Non-Separated  & $\frac{3}{2}(k-1)n-\frac{1}{2}(2k-1-\{n+1\}_k)(k-1+\{n+1\}_k)$ & New, see Theorem~\ref{thm:nonsep}\\
Non-Nested  & $\in[(k-1)n,(k-1)n+\binom{k-1}{2}]$ & New, see Theorem~\ref{thm:non-nested}\\
Non-crossing   & $(k-1)n$ & \cite{Kupitz}\\
\hline
\end{tabular}
\egroup
\caption{The values in the table hold for every $n,k\in\mathbb N$ with $n\ge 2k$.
Here $T(n+1,k)$ denotes the $k$-partite Tur\'an graph on $n+1$ vertices and $\{n+1\}_k$ denotes the residue of $n+1$ divided by $k$. }
\label{table:1}
\end{table}

We also obtain sharp results for the Tur\'an number of the alternating path (see Theorem~\ref{thm:altpath}) and for $ex_<(n,\text{Cross-}\mathcal M_k,\text{Sep-}\mathcal M_k)$ (i.e., when we forbid a crossing {\it and} a separated $k$-matching).

Recently, there has been a lot of research on ordered Ramsey numbers \cite{bckk,cfls,tardos}, where a monochromatic copy of a small ordered target graph is searched in a large 2-edge-coloured large host graph.
We show that Tur\'an-type results for alternating paths and non-nested matchings can yield better bounds for the analogous Ramsey-type problems.



\subsection{Previous work on ordered matchings}
A {\em convex geometric graph} is a graph drawn in the plane such that its vertices are in convex position and the edges are straight line segments.
Observe  that $ex_<(n,\text{Cross-}\mathcal M_k)$ is equal to the maximum number of edges in a geometric convex graph that does not contain $k$ pairwise crossing independent edges. 
Similarly, $ex_<(n,\neg\text{Cross-}\mathcal M_k)$ is the maximum number of edges in a geometric convex graph that  
does not contain $k$ pairwise non-crossing independent edges.
These functions have already been determined.

\begin{theorem}[Kupitz \cite{Kupitz}]
For every $n,k\in\mathbb N$ with $n\ge 2k$, we have 
$$ex_<(n,\neg\text{Cross-}\mathcal M_k)=(k-1)n.$$    
\end{theorem}

\begin{theorem}[Capoyleas and Pach \cite{CapoyleasP}]\label{thm:crossing}
For every $n,k\in\mathbb N$ with $n\ge 2k$, we have
$$ex_<(n,\text{Cross-}\mathcal M_k)=2(k-1)n-(2k-1)(k-1).$$    
\end{theorem}

The functions $ex_<(n,\text{Sep-}\mathcal M_k)$ and $ex_<(n,\text{Nest-}\mathcal M_k)$ have also been considered, under different terminology.
A {\it stack} (resp. {\it queue} and {\it arch}) is a set of edges that are non-crossing (resp.non-nested and non-separated).
A {\it $k$-stack layout} of an ordered graph~$G$, is a partition of the edge set $E(G)$ into $k$ stacks; the notions of {\it $k$-queue layout} and {\it $k$-arch layout} are defined analogously.
In particular,~$G$ admits a $k$-queue layout (resp. a $k$-arch layout) if and only if it does not contain a nested (resp. separated) $(k+1)$-matching, see e.g.~\cite[Lemmata~1 and~2]{layout}. 
Dujmovi\'c and Wood \cite{layout} determined the maximum number of edges in an ordered graph with a $k$-queue layout or a $k$-arch layout~\cite[Lemmata~8 and~10]{layout}, thus obtaining the following.

\begin{theorem}[Dujmovi\'c and Wood~\cite{layout}]\label{thm:nested}
For every $n,k\in\mathbb N$ with $n\ge 2k$, we have
$$ex_<(n,\text{Nest-}\mathcal M_k)=2(k-1)n - (k-1)(2k - 1).$$    
\end{theorem}

In the next result, $T(n+1,k)$ denotes the $k$-partite Tur\'an graph on $n+1$ vertices, that is, the complete $k$-partite graph with classes of size as equal as possible.

\begin{theorem}[Dujmovi\'c and Wood~\cite{layout}]\label{dwsepar}
For every $n,k\in\mathbb N$ with $n\ge 2k$, we have
$$ex_<(n,\text{Sep-}\mathcal M_k)=e(T(n+1,k))-k+1.$$
\end{theorem}
Note that Theorem~\ref{dwsepar} is phrased slightly less precisely in~\cite{layout}.
For completeness, a self-contained proof of Theorem~\ref{dwsepar} can be found in the appendix.

\subsection{Organisation of the paper and notation}
The remainder of the paper is organised as follow.
In Section~\ref{sec:newresults}, we state and prove our new Tur\'an-type results.
These include Tur\'an-type results for non-separated and non-nested matchings.
We also prove exact Tur\'an-type results for alternating path and $ex_<(n,\text{Cross-}\mathcal M_k,\text{Sep-}\mathcal M_k)$.
In Section~\ref{sec:appRamsey}, we show how our Tur\'an-type results for alternating paths and non-nested matchings yield consequences for Ramsey-type questions.
Section~\ref{sec:conclusion} includes some final remarks and future research directions. 

\smallskip

\noindent{\bf Notation.} Suppose that $G$ is a graph with vertex set $[n]$. 
We write $e(G)$ for the number of edges in $G$.
For any $I\subseteq [n]$, 
$G[I]$ denotes 
the subgraph of $G$ induced by the vertices/points of $I$. 
Let~$E(X,Y)$ denote the edges between vertices~$X$ and~$Y$, and $e(X,Y)$ the number of these edges. 

\section{New Tur\'an-type results}\label{sec:newresults}

\subsection{The Tur\'an number of non-separated matchings}

\begin{theorem}\label{thm:nonsep}
For every $n,k\in\mathbb N$ with $n\ge k$, 
we have
$$ex_<(n,\neg\text{Sep-}\mathcal M_k)= \left\lceil\frac{n-2k+1}{k}\right\rceil\left(\binom{2k-1}{2}-\binom{k-1}{2}\right)+\binom{2k-1-\{n+1\}_k}{2},$$ 
where we write $\{N\}_k$ to denote the residue of $N$ divided by $k$.
\end{theorem}


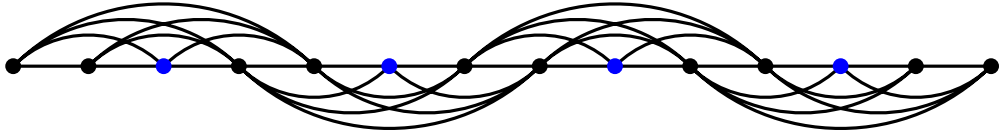
\begin{figure}[h!]
\centering
\begin{tikzpicture}
\draw[very thick] (0,0) -- (13,0);

\draw[very thick] 
(0,0) arc (135:45:1.414)
(1,0) arc (135:45:1.414)
(2,0) arc (135:45:1.414)
(0,0) arc (135:45:2.121)
(1,0) arc (135:45:2.121)
(0,0) arc (135:45:2.828);

\draw[very thick] 
(3,0)+(0,0) arc (-135:-45:1.414)
(3,0)+(1,0) arc (-135:-45:1.414)
(3,0)+(2,0) arc (-135:-45:1.414)
(3,0)+(0,0) arc (-135:-45:2.121)
(3,0)+(1,0) arc (-135:-45:2.121)
(3,0)+(0,0) arc (-135:-45:2.828);

\draw[very thick] 
(6,0)+(0,0) arc (135:45:1.414)
(6,0)+(1,0) arc (135:45:1.414)
(6,0)+(2,0) arc (135:45:1.414)
(6,0)+(0,0) arc (135:45:2.121)
(6,0)+(1,0) arc (135:45:2.121)
(6,0)+(0,0) arc (135:45:2.828);

\draw[very thick] 
(9,0)+(0,0) arc (-135:-45:1.414)
(9,0)+(1,0) arc (-135:-45:1.414)
(9,0)+(2,0) arc (-135:-45:1.414)
(9,0)+(0,0) arc (-135:-45:2.121)
(9,0)+(1,0) arc (-135:-45:2.121)
(9,0)+(0,0) arc (-135:-45:2.828);

\fill (0,0) circle (3pt); 
\fill (1,0) circle (3pt);    
\fill[blue] (2,0) circle (3pt);  
\fill (3,0) circle (3pt);  
\fill (4,0) circle (3pt);  
\fill[blue] (5,0) circle (3pt);  
\fill (6,0) circle (3pt); 
\fill (7,0) circle (3pt); 
\fill[blue] (8,0) circle (3pt); 
\fill (9,0) circle (3pt); 
\fill (10,0) circle (3pt); 
\fill[blue] (11,0) circle (3pt); 
\fill (12,0) circle (3pt); 
\fill (13,0) circle (3pt); 

\end{tikzpicture}
\caption{An extremal construction for Theorem~\ref{thm:nonsep} with $n=14$ and $k=3$.}
\label{fig:non-separated}
\end{figure}

\begin{proof}
    First we exhibit an extremal example.
    We construct it recursively.
    Start with a clique with vertex set $[2k-1-\{n+1\}_k]$.
    Trivially, this does not contain a $k$-matching.
    Suppose we have constructed a graph~$G_t$ with vertex set $[2k-1-\{n+1\}_k+kt]$, for some integer $t\ge0$.
    We define~$G_{t+1}$ to be the graph with vertex set $[2k-1-\{n+1\}_k+k(t+1)]$ and edge set $E(G_{t+1})=E(G_t)\cup S_t$ where
    $S_t$ is the set of edges induced by the last $2k-1$ vertices.
    
    We have
    $e(G_{t+1})=e(G_t)+\binom{2k-1}{2}-\binom{k-1}{2}$.
    It is easy to check that, for $t_0=\lceil(n-2k+1)/k\rceil$, the graph $G_{t_0}$ has $n$ vertices and precisely the required number of edges.
    Suppose for a contradiction that $G_t$ contains a non-separated $k$-matching $M$, and take $t$ to be minimal with respect to this property.
    In particular, we have $t\ge1$.
    If $M$ contains an edge lying in $V(G_{t+1})\setminus V(G_t)$, then such an edge is separated to any edge in $V(G_t)$. In that case, $M$ must lie entirely in $V(G_{t+1})\setminus V(G_t)$, which is not possible as this is a set of only $k$ vertices.
    By minimality, $M\cap V(G_{t+1})$ is non-empty and, by the previous observation, all such edges must be incident to the last $k-1$ vertices of $G_t$. 
    As $M$ contains no separated edges, this implies all edges of $M$ intersect the last $k-1$ vertices of $G_t$, a contradiction.

It remains to prove an appropriate upper bound for $ex_<(n,\neg\text{Sep-}\mathcal M_k)$.
The following key claim states that there exists a graph~$G$ attaining $ex_<(n,\neg\text{Sep-}\mathcal M_k)$ such that no edge of~$G$ lies above a non-edge.
Note that this property is shared by the extremal example described above.
The proof of the claim relies on a shifting argument.
    \begin{claim}\label{claim:key}
        There exists a graph~$G$ with $V(G)=[n]$ and $e(G)=ex_<(n,\neg\text{Sep-}\mathcal M_k)$, which does not contain a non-separated $k$-matching and, for every $x\le a<b\le y$, if $xy\in E(G)$ then $ab\in E(G)$.
    \end{claim}
    \begin{proofclaim}
        Pick a graph $G$ with $V(G)=[n]$ and $e(G)=ex_<(n,\neg\text{Sep-}\mathcal M_k)$ satisfying the following.
        The graph~$G$ does not contain a non-separated $k$-matching and, subject to this, the sum
        $$\sum_{xy\in E(G)}|x-y|$$
        is as small as possible.
        We show that~$G$ satisfies the required properties of the claim.
        Suppose for a contradiction that there exist $x\le a<b\le y$ such that $xy\in E(G)$ and $ab\notin E(G)$.
        If we select $a$ and $b$  so that $|a-b|$ is as large as possible, then either $(a{-}1)b\in E(G)$ or $a(b{+}1)\in E(G)$.
        By symmetry, we may assume that $a(b{+}1)\in E(G)$ and $ab\notin E(G)$.

        Let~$G^*$ be the graph obtained by applying the following procedure on~$G$.
        For every $v\in[n]$ with $v<b$, if $vb\notin E(G)$ and $v(b{+}1)\in E(G)$, then replace the edge $v(b{+}1)$ with the edge $vb$.
        Firstly, it is easy to see that $G^*$ and $G$ have the same number of edges.
        In particular, $e(G^*)=ex_<(n,\neg\text{Sep-}\mathcal M_k)$.
        Furthermore, the procedure always replaces edges with shorter ones, and we know that at least one replacement has occurred (namely, the edge $a(b{+}1)$ has been replaced by the edge $ab$).
        Therefore, we have
        $$\sum_{xy\in E(G^*)}|x-y|<\sum_{xy\in E(G)}|x-y|.$$
        By minimality of~$G$, we conclude that~$G^*$ must contain a non-separated $k$-matching consisting of, edges $e_1,\dots,e_k$ say.
        
        Since $G$ does not contain a non-separated $k$-matching, there exists $i\in[k]$ such that~$e_i$ is not an edge of~$G$.
        We may assume $i=1$.
        In particular, we must have $e_i=vb$ for some $v<b$ where $v(b{+}1)\in E(G)$ and $vb\notin E(G)$.
        Note that the edges $e_2,\dots,e_k$ are not incident to $b$ since they are disjoint from~$e_1$.
        If $e_2,\dots,e_k$ are not incident to $b+1$ either, then they are edges in~$G$.
        However, in that case, $v(b{+}1),e_2,\dots,e_k$ is a non-separated $k$-matching in~$G$, a contradiction.
        Hence, we may assume that~$e_2$ is incident to~$b+1$, say $e_2=u(b{+}1)$.
        Since $e_1$ and $e_2$ are not separated and $v<b$, we must have $u<b$ too.
        Also, since $e_2$ is an edge in $G^*$, we must have that $ub\in E(G)$, otherwise $u(b{+}1)$ would have been replaced by~$ub$.
        Finally, observe that $v(b{+}1),ub,e_3,\dots,e_k$ are all edges in~$G$ and they form a non-separated $k$-matching, a contradiction.
        \qedclaim        
    \end{proofclaim}

    In the next claim, we prove that a graph attaining $ex_<(n,\neg\text{Sep-}\mathcal M_k)$ must contain a minimal missing edge of length approximately $k$.
    Note that this is also a property shared with the extremal construction (see the missing edge between consecutive blue points in Figure~\ref{fig:non-separated}e).

    \begin{claim}\label{below_comp}
        Let~$G$ be a graph with $V(G)=[n]$ and $e(G)=ex_<(n,\neg\text{Sep-}\mathcal M_k)$ which does not contain a non-separated $k$-matching.
        If $n\ge2k$, there exists some $x\in[n]$ such that $x(x+k)$ is a missing edge and the graph $G[\{x+1,\dots,x+k-1\}]$ is complete.
    \end{claim}
    \begin{proofclaim}
        By the assumptions, we known~$G$ must have at least one missing edge and adding any missing edge must create a non-separated $k$-matching.
        Let $uv$ be a missing edge, say $u<v$, which minimises $v-u$.
        In particular, all edges, except $uv$ must be present in $G[u,u+1,\dots,v]$.
        By maximality of~$G$, it follows that $G+uv$ contains a non-separated $k$-matching $\mathcal M$, that contains the edge $uv$.
        Suppose that an edge $ab$ of the matching $\mathcal M$ lies below $uv$, say $u<a<b<v$.
        Now $\mathcal M\cup\{ub,av\}\setminus\{ab\}$ is a non-separated $k$-matching in $G$, contradiction.
        Suppose that there is a vertex~$w$ with $u<w<v$ which is not saturated by the matching $\mathcal M$.
        Let $\mathcal M_1:=\mathcal M\cup\{uw\}\setminus\{uv\}$ and $\mathcal M_2:=\mathcal M\cup\{wv\}\setminus\{uv\}$.
        Both $\mathcal M_1$ and $\mathcal M_2$ are $k$-matchings in $G$.
        Therefore, by assumption, both of them contain a separated pair of edges.
        In particular, there exists an edge $e\in\mathcal M\setminus\{uv\}$ such that $e$ and $uw$ are separated.
        If $e$ lies to the left of $uw$, then $e$ is also separated to $uv$, contradicting the assumption that $\mathcal M$ is non-separated.
        Hence, $e$ lies to the right of $w$.
        Similarly, there exists an edge $e'\in\mathcal M\setminus\{uv\}$, which lies to the left of $w$.
        Again, this is a contradiction, since $\mathcal M$ was non-separated.

        We conclude that the interval $[u+1,v-1]$ must have size precisely $k-1$.
        Hence $v=u+k$, and we are done.\qedclaim
    \end{proofclaim}

    We prove Theorem~\ref{thm:nonsep} by induction on $n$.
    The theorem trivially holds if $k\le n\le 2k-1$, since the extremal construction is simply the complete graph.
    This is the base step.
    Suppose instead that $n\ge 2k$.
    Let~$G$ be a graph with $V(G)=[n]$ and $e(G)=ex_<(n,\neg\text{Sep-}\mathcal M_k)$, which does not contain a non-separated $k$-matching and satisfied the properties of Claim~\ref{claim:key}.
    Let~$x\in[n]$ satisfy the conditions of Claim~\ref{below_comp} and define graphs $G_1:=G[1,2,\dots,x{+}k-1]$ and $G_2:=G[x{+}1,\dots,n]$.
    By the inductive hypothesis, since~$G_1$ and~$G_2$ have at least~$k$ vertices each, we have the following upper bounds on $e(G_1)$ and $e(G_2)$ (where we also use the equality $\lceil N/k\rceil=(N+\{N\}_k)/k$ for any positive integer $N$).


    $$e(G_1)\le\frac{|G_1|-2k+1+\{|G_1|+1\}_k}{k}\left(\binom{2k-1}{2}-\binom{k-1}{2}\right)+\binom{2k-1-\{|G_1|+1\}_k}{2}$$
    $$e(G_2)\le\frac{|G_2|-2k+1+\{|G_2|+1\}_k}{k}\left(\binom{2k-1}{2}-\binom{k-1}{2}\right)+\binom{2k-1-\{|G_2|+1\}_k}{2}$$
    Observe $e(G)=e(G_1)+e(G_2)-e(G_1\cap G_2)+e([x],[x{+}k,n])$.
    In particular, we have 
    $$e(G)=e(G_1)+e(G_2) -\binom{k-1}{2}$$ 
    since $G_1\cap G_2$ is complete (as $x(x+k)$ is minimal), and $e([x],[x{+}k,n])=0$ by Claim~\ref{claim:key} (as $x(x+k)$ is missing, there are no edges above it).
    Combining these equalities with the bounds on~$e(G_1)$ and~$e(G_2)$, we obtain
    \begin{align}
    e(G)= e(G_1)+e(G_2) -\binom{k-1}{2}\le &\;\frac{n-2k+1+\{|G_1|+1\}_k+\{|G_2|+1\}_k}{k}\left(\binom{2k-1}{2}-\binom{k-1}{2}\right) \notag\\ 
     &\;+f(\{|G_1|+1\}_k,\{|G_2|+1\}_k)-\binom{2k-1}{2}\label{eq:e(G)}
    \end{align}
    where we used $|G_1|+|G_2|=n+k-1$ and set $f(n_1,n_2):=\binom{2k-1-n_1}{2}+\binom{2k-1-n_2}{2}$ for any integers~$n_1$ and~$n_2$.

    \begin{claim}\label{claim:function}
         For integers $1\le n_1\le n_2\le 2k-2$ we have $f(n_1,n_2)\le f(n_1-1,n_2+1)$.
    \end{claim}
    \begin{proofclaim}
    We have
    \begin{align*}f(n_1-1,n_2+1)&=\binom{2k-1-n_1+1}{2}+\binom{2k-1-n_2-1}{2}\\
    &=\binom{2k-1-n_1}{2}+(2k-1-n_1)+\binom{2k-1-n_2}{2}-(2k-1-n_2-1)\\
    &=f(n_1,n_2)+(n_2-n_1+1)>f(n_1,n_2).
    \end{align*}\qedclaim
    \end{proofclaim}
    
    Since $|G_1|+1+|G_2|+1=n+k+1$, we either have $\{|G_1|+1\}_k+\{|G_2|+1\}_k=\{n+1\}_k$ or $\{|G_1|+1\}_k+\{|G_2|+1\}_k=\{n+1\}_k+k$.
    In the former case, we have $f(\{|G_1|+1\}_k,\{|G_2|+1\}_k)\le f(0,\{n+1\}_k)$ by Claim~\ref{claim:function}.
    Combining this with inequality~\eqref{eq:e(G)}, we obtain
    \begin{align}\label{eq:case1}
    e(G)\le\frac{n-2k+1+\{n+1\}_k}{k}\left(\binom{2k-1}{2}-\binom{k-1}{2}\right)+f(0,\{n+1\}_k)-\binom{2k-1}{2}.
    \end{align}
    On the other hand, if $\{|G_1|+1\}_k+\{|G_2|+1\}_k=\{n+1\}_k+k$ then $f(\{|G_1|+1\}_k,\{|G_2|+1\}_k)\le f(k,\{n+1\}_k)$ again by Claim~\ref{claim:function}.
    Combining this with inequality~\eqref{eq:e(G)}, we obtain
    \begin{align}\label{eq:case2}
    e(G)\le\frac{n-2k+1+\{n+1\}_k+k}{k}\left(\binom{2k-1}{2}-\binom{k-1}{2}\right)+f(k,\{n+1\}_k)-\binom{2k-1}{2}.
    \end{align}
    Finally, it is easy to show that the right hand-sides of both~\eqref{eq:case1} and~\eqref{eq:case2} are equal to
    $$\frac{n-2k+1+\{n+1\}_k}{k}\left(\binom{2k-1}{2}-\binom{k-1}{2}\right)+\binom{2k-1-\{n+1\}_k}{2}$$ 
    as required.

\end{proof}

\subsection{The Tur\'an number of (strongly) non-nested matchings}

The question for non-nested matching is apparently the most difficult to settle.
Before proceeding further, we introduce a special type of non-nested matching.

A $k$-matching within an ordered graph is {\it strongly non-nested} if it is a disjoint union of crossing matchings such that two edges cross if and only if they belong to the same crossing matching. 

\begin{figure}[!ht]
    \centering
    \includegraphics[width=0.6\linewidth]{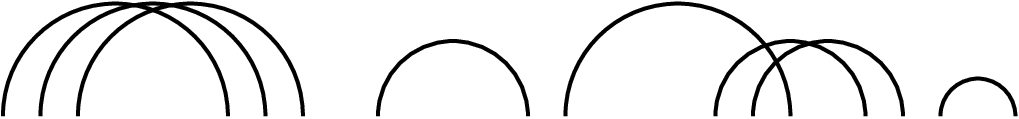}
    \caption{A typical strongly non-nested matching.}
    \label{islands}
\end{figure}

We denote a strongly non-nested $k$-matching as $\neg\text{Nest*-}\mathcal M_k$.
Since a strongly non-nested matching is non-nested, we immediately have

$$ex_<(n,\neg\text{Nest-}\mathcal M_k)\le ex_<(n,\neg\text{Nest*-}\mathcal M_k).$$

We are able to give lower and upper bounds for the Tur\'an numbers above that differ by a quadratic polynomial in $k$.
Thus, this resolves the question asymptotically in $n$, for $k$ fixed.

\begin{theorem}\label{thm:non-nested}
    For every $n,k\in\mathbb N$ with $n\ge2k$, we have 
    $$(k-1)n\le ex_<(n,\neg\text{Nest-}\mathcal M_k)\le ex_<(n,\neg\text{Nest*-}\mathcal M_k)\le(k-1)n+\binom{k-1}{2}.$$    
\end{theorem}

In fact, we conjecture that the lower bound in Theorem~\ref{thm:non-nested} is optimal for strongly non-nested matching, and thus for non-nested matching.

\begin{restatable}{conjecture}{nnconj}\label{conj:nn}
    For every $n,k\in\mathbb N$ with $n\ge 2k$, we have 
    $$ex(n,\neg\text{Nest-}\mathcal M_k)=ex(n,\neg\text{Nest*-}\mathcal M_k)=(k-1)n.$$  
\end{restatable}

There are many graphs achieving the lower bound in Theorem~\ref{thm:non-nested} (and Conjecture~\ref{conj:nn}).
Indeed, given an ordered graph $G$ containing no non-nested $k$-matching and with $e(G)=(k-1)|G|$, consider the following two operations.

\smallskip

{\bf 1.} If $G$ has vertex set $[n]$, add a new vertex $n+1$ and all edges $(i,n+1)$ for $i\in[k-1]$.

{\bf  2.} If $G$ has vertex set $[n]$ then relabel vertex $i$ with $n-i$ for every $i\in[n]$.

\smallskip

None of the two operations create a non-nested $k$-matching, and the resulting graph $G$ satisfies $e(G)=(k-1)|G|$.
Hence, starting with a clique on $2k-1$ vertices and performing {\bf 1} and {\bf 2} in any order, yields many such graphs, see for example Figure~\ref{fig:non-nested1}.

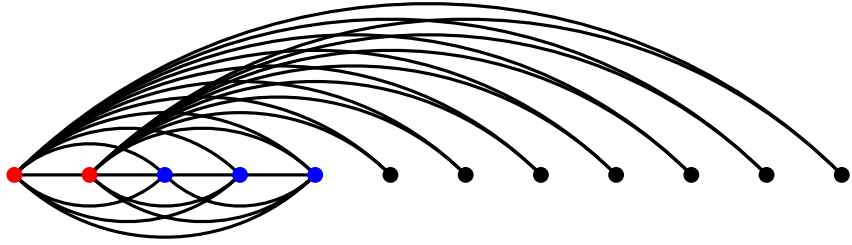
\begin{figure}[h!]
\centering
\begin{tikzpicture}

\draw[very thick] (0,0) -- (4,0);
\draw[very thick] 
(0,0) arc (-135:-45:1.414)
(1,0) arc (-135:-45:1.414)
(2,0) arc (-135:-45:1.414)
(0,0) arc (-135:-45:2.121)
(1,0) arc (-135:-45:2.121)
(0,0) arc (-135:-45:2.828);

\foreach \i in {2,...,11}{
\draw[very thick] (0,0) arc (135:45:\i/1.414);}

\foreach \i in {3,...,10}{
\draw[very thick] (1,0) arc (135:45:\i/1.414);}

\fill[red] (0,0) circle (3pt); 
\fill[red] (1,0) circle (3pt);    
\fill[blue] (2,0) circle (3pt);  
\fill[blue] (3,0) circle (3pt);  
\fill[blue] (4,0) circle (3pt);  
\fill(5,0) circle (3pt);  
\fill (6,0) circle (3pt); 
\fill (7,0) circle (3pt); 
\fill (8,0) circle (3pt); 
\fill (9,0) circle (3pt); 
\fill (10,0) circle (3pt); 
\fill (11,0) circle (3pt); 
\end{tikzpicture}
\caption{A construction achieving the lower bound of Theorem~\ref{thm:non-nested} with $k=3$, obtained by starting with a $5$-vertex clique and repeatedly applying step {\bf 1}. 
}
\label{fig:non-nested1}
\end{figure}

The following is yet another construction achieving the same bound, which is not obtained by the operations above.
Note that further examples can be obtained by combining this new construction with operations {\bf 1} and {\bf 2}.

\smallskip

{\bf Construction.} For any $n,k\in\mathbb N$ with $n\ge3(k-1)$, let~$G$ be the following graph with vertex set~$V(G)=[n]$.
Fix any $k-1$ vertices in $[k,n-k+1]$ and add all edges incident to them. There are $(k-1)(n-1)-\binom{k-1}{2}$ of these edges.
Also add all edges $xy$ with $|x-y|\ge n-k+1$.
There are $k-1+k-2+\dots+1=\binom{k}{2}$ such edges.
See Figure~\ref{fig:non-nested2} for a drawing of this graph.
Now, $e(G)=(k-1)n$.
Suppose to the contrary that~$G$ contains a non-nested $k$-matching $M$.
Observe that $M$ does not contain any edge $xy$ such that $|x-y|\ge n-k+1$.
Indeed, if $x<y$, then each of the remaining $k-1$ edges of $M$ would be incident to $\{v\in[n]:v<x \text{ or } v>y\}$, but there are only $k-2$ such vertices.
Hence, all edges of $M$ must be incident to the $k-1$ vertices in $[k,n-k+1]$, which is again a contradiction.

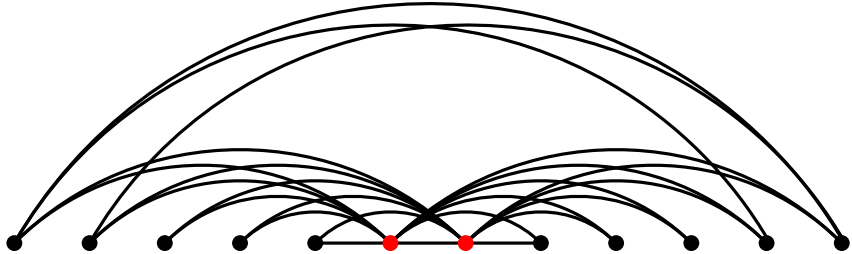
\begin{figure}[h!]
\centering
\begin{tikzpicture}
\draw[very thick] (0,0) arc (150:30:6.37);
\draw[very thick] (0,0) arc (150:30:5.8);
\draw[very thick] (1,0) arc (150:30:5.8);

\foreach \i in {2,...,6}{
\draw[very thick] (5,0) arc (135:45:\i/1.414);}
\foreach \i in {2,...,5}{
\draw[very thick] (6,0) arc (135:45:\i/1.414);}

\foreach \i in {2,...,5}{
\draw[very thick] (5,0) arc (45:135:\i/1.414);}
\foreach \i in {2,...,6}{
\draw[very thick] (6,0) arc (45:135:\i/1.414);}

\draw[very thick] (4,0) -- (7,0);

\fill (0,0) circle (3pt); 
\fill (1,0) circle (3pt);    
\fill (2,0) circle (3pt);  
\fill (3,0) circle (3pt);  
\fill (4,0) circle (3pt);  
\fill[red] (5,0) circle (3pt);  
\fill[red] (6,0) circle (3pt); 
\fill (7,0) circle (3pt); 
\fill (8,0) circle (3pt); 
\fill (9,0) circle (3pt); 
\fill (10,0) circle (3pt); 
\fill (11,0) circle (3pt); 
\end{tikzpicture}
\caption{A construction achieving the lower bound of Theorem~\ref{thm:non-nested} with~$k=3$ and vertex set~$[12]$.
An edge~$xy$ is included if either $x$ and/or $y$ is red, or if $|x-y|\in\{10,11\}$.}
\label{fig:non-nested2}
\end{figure}

\subsubsection{Proof of the upper bound of Theorem~\ref{thm:non-nested}}

The following result provides an upper bound for the number of long edges in a graph with no crossing $k$-matching.
It is essentially a corollary of Theorem~\ref{thm:crossing}.

\begin{corollary}\label{col:crossing}
    Let $n,k\in\mathbb N$ with $n\ge2k$.
    If $G$ is a graph with $V(G)=[n]$ and $G$ does not contain a crossing $k$-matching, then
    $$|\{xy\in E(G):|x-y|\ge k\}|\le(k-1)n-\binom{2k-1}{2}+\binom{k}{2}.$$
\end{corollary}

\begin{proof}
    We start by repeatedly adding any missing edge to $G$ without creating a crossing $k$-matching, as long as possible.
    Observe that adding a missing edge $xy$ with $|x-y|\le k-1$ never creates a crossing $k$-matching.
    Therefore, at the end of this procedure, the graph $G$ does not contain a crossing $k$-matching and $xy\in E(G)$ for every distinct $x,y\in[n]$ with $|x-y|\le k-1$.
        
    By Theorem~\ref{thm:crossing}, we have $e(G)\le ex_<(n,\text{Cross-}\mathcal M_k)=2(k-1)n-\binom{2k-1}{2}$.
    It follows that

    \begin{align*}
    & |\{xy\in E(G):|x-y|\ge k\}|=e(G)-|\{xy\in E(G):|x-y|\le k-1\}|\\
    =\; & e(G)-\sum_{i=1}^{k-1}(n-i)\le\left(2(k-1)n-\binom{2k-1}{2}\right)-\left((k-1)n-\binom{k}{2}\right)\\
    =\; & (k-1)n-\binom{2k-1}{2}+\binom{k}{2}.
    \end{align*}
\end{proof}

On the other hand, a simple pigeonhole argument yields an upper bound for the number of short edges in an ordered graph with no strongly non-nested $k$-matching.

\begin{lemma}\label{lem:non-nested}
    Let $n,k\in\mathbb N$ with $n\ge2k$.
    If $G$ is a graph with $V(G)=[n]$ and $G$ does not contain a strongly non-nested $k$-matching, then
    $$|\{xy\in E(G):|x-y|\le k-1\}|\le2(k-1)^2.$$
\end{lemma}

\begin{proof}
    Let~$K$ denote the complete graph on vertex set~$[n]$. 
    Observe that, for every $\ell\in[n]$, the edge set $\{xy\in E(K):|x-y|=\ell\}$ can be partitioned into two strongly non-nested $k$-matchings $M_1$ and $M_2$.
    Therefore, we have $|\{xy\in E(G):|x-y|=\ell\}|\le 2(k-1)$ as otherwise, by the pigeonhole principle, either $E(G)\cap M_1$ or $E(G)\cap M_2$ has more than $k$ edges.
    Hence $G$ contains a strongly non-nested $k$-matching, a contradiction.
    This yields the required upper bound.
\end{proof}

Combining Corollary~\ref{col:crossing} and Lemma~\ref{lem:non-nested} immediately yields Theorem~\ref{thm:non-nested}.

\begin{proof}[Proof of Theorem~\ref{thm:non-nested}]
    The lower bound follows from the Construction given earlier in the section.
    For the upper bound, it suffices to show that if $G$ is a graph with $V(G)=[n]$ and $G$ does not contain a strongly non-nested $k$-matching, then $e(G)\le(k-1)n+\binom{k-1}{2}$.
    
    Observe that $G$ does not contain a crossing $k$-matching.
    Thus, by Corollary~\ref{col:crossing}, we have  
    $$|\{xy\in E(G):|x-y|\ge k\}|\le(k-1)n-\binom{2k-1}{2}+\binom{k}{2}.$$
    On the other hand, by Lemma~\ref{lem:non-nested}, we have
    $$|\{xy\in E(G):|x-y|\le k-1\}|\le2(k-1)^2.$$
    It follows that
    \begin{align*}
    e(G) & =|\{xy\in E(G):|x-y|\ge k\}|+|\{xy\in E(G):|x-y|\le k-1\}|\\
    & \le(k-1)n-\binom{2k-1}{2}+\binom{k}{2}+2(k-1)(k-1)\\
    & =(k-1)n+\binom{k-1}{2}.
    \end{align*}
\end{proof}

\subsection{The Tur\'an number of crossing versus separated matchings}

The following theorem establishes the maximum number of edges in an ordered graph, which does not contain a crossing or separated $k$-matching.

The important point we make here is that the upper bound for non-nested matching is smaller than the construction here, for $n\ge 2k^2$. 
In other words, there is no hope proving the tight non-nested extremal number if we only exclude 2 particular, although opposite cases.
The proof closely resembles that of Theorem~\ref{thm:non-nested}.

\begin{theorem}\label{thm:cross,sep}
    For every $n,k\in\mathbb N$ with $n\ge2k$ we have $$ex_<(n,\text{Cross-}\mathcal M_k,\text{Sep-}\mathcal M_k)\le(k-1)n-\binom{2k-1}{2}+\binom{k}{2}+(k-1)\left(\binom{k+1}{2}-1\right).$$    
    Furthermore, equality holds provided $n\ge 2k^2$.
\end{theorem}

\begin{proof}
    Firstly, we show that if $G$ is a graph with $V(G)=[n]$ and $G$ does not contain any crossing or separated $k$-matching, then $e(G)\le(k-1)n-\binom{2k-1}{2}+\binom{k}{2}+(k-1)(\binom{k+1}{2}-1)$.
  
    As $G$ does not contain a crossing $k$-matching, we can upper bound the number of long edges by Corollary~\ref{col:crossing} as follows:
    $$|\{xy\in E(G):|x-y|\ge k\}|\le(k-1)n-\binom{2k-1}{2}+\binom{k}{2}.$$
    Next, we bound the number of short edges from above.
    For a fixed $\ell\in\mathbb N$ with $\ell\le k-1$, observe that the edges $\{xy\in E(K_n):|x-y|=\ell\}$ of the complete graph $K_n$ with vertex set $[n]$ can be partitioned into $\ell+1$ separated matchings $M_1,\dots,M_{\ell+1}$.
    Since $G$ does not contain a separated $k$-matching, we conclude that $G$ contains at most $(k-1)$ edges from each $M_i$, and thus $|\{xy\in E(G):|x-y|=\ell\}|\le(\ell+1)(k-1)$.
    
    It follows that
    \begin{align*}
    e(G) & =|\{xy\in E(G):|x-y|\ge k\}|+|\{xy\in E(G):|x-y|\le k-1\}|\\
    & \le(k-1)n-\binom{2k-1}{2}+\binom{k}{2}+\sum_{\ell=1}^{k-1} (\ell+1)(k-1)\\
    & =(k-1)n-\binom{2k-1}{2}+\binom{k}{2}+(k-1)(\binom{k+1}{2}-1).
    \end{align*}

To see that equality holds for $n\ge 2k^2$, consider the graph $G$ on vertex set  $[n]$. Let $v_i=ik$ for $i=1, 2, \ldots, k-1$ and let 
$$E(G)=\{\  xy\ |\  x\le k-1, \ |x-y|\ge k\ \}
\cup\{\ xy\ |\ |x-y|\le k-1 \ {\mbox{and for some}}\ i, x\le v_i\le y\ \}.$$
Now
$e(G)=(k-1)n-\binom{2k-1}{2}+\binom{k}{2}+(k-1)(\binom{k+1}{2}-1)$.    
\end{proof}

\subsection{The Tur\'an number of alternating paths}

For any $t\in\mathbb N$, the alternating $t$-path $\text{Alt-}P_t$ is an ordered path on $t$ vertices such that every pair of edges (including incident pairs) are nested, see Figure~\ref{fig:altpath}.
Two edges $xc$ and $yc$ with a common end-vertex $c$ are nested if and only if $x<y<c$ or $y<x<c$.

\begin{figure}[h!]
\centering
\begin{tikzpicture}
\draw[very thick] (0,0) arc (135:45:7.06);
\draw[very thick] (0,0) arc (135:45:5.65);
\draw[very thick] (2,0) arc (135:45:4.24);
\draw[very thick] (2,0) arc (135:45:2.82);
\draw[very thick] (4,0) arc (135:45:1.41);

\fill (0,0) circle (3pt)   
(2,0) circle (3pt) 
(4,0) circle (3pt)
(6,0) circle (3pt)
(8,0) circle (3pt)
(10,0) circle (3pt); 
\end{tikzpicture}
\caption{An alternating path on~$6$ vertices.}
\label{fig:altpath}
\end{figure}
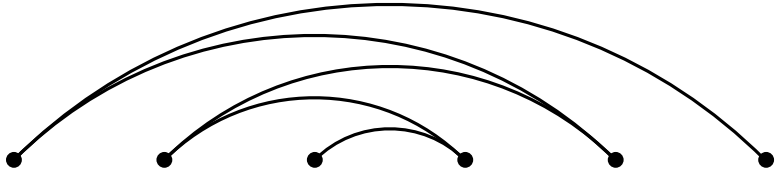

Recall that the maximum number of edges in an ordered graph with no nested $k$-matching was determined in Theorem~\ref{thm:nested}. Namely 
$$ex_<(n,\text{Nest-}\mathcal M_k)=2(k-1)n-(k-1)(2k-1).$$ 
In fact, our proof yields the same result for alternating paths with just minor modifications.

\begin{theorem}\label{thm:altpath}
For every $n,k\in\mathbb N$ with $n\ge 2k$ we have\footnote{For simplicity, we insist that the vertices of the alternating path are even in number.}
$$ex_<(n,\text{Alt-}P_{2k})=2(k-1)n-(k-1)(2k-1).$$    
\end{theorem}
\begin{proof}
Let~$H$ be a graph with vertex set~$[n]$, where $xy$ is an edge if and only if $|y-x|\le2k-2$.
Given a nested $k$-matching on~$[n]$, note that the edge $uv$ containing all other edges of the matching must satisfy $|u-v|\ge2k-1$.
Therefore, $H$ does not contain a nested $k$-matching.
Since an alternating $2k$-path contains a nested $k$-matching, we conclude that $H$ does not contain an alternating $2k$-path either.
It is easy to check that $e(H)=2(k-1)n-(k-1)(2k-1)$,
hence we have $ex_<(n,\text{Alt-}P_{2k})\ge 2(k-1)n-(k-1)(2k-1)$.
It remains to show that $ex_<(n,\text{Alt-}P_{2k})\le2(k-1)n-(k-1)(2k-1)$.

Let~$G$ be an arbitrary graph with vertex set $[n]$, which does not contain a copy of~$\text{Alt-}P_{2k}$ and with $e(G)=ex_<(n,\text{Alt-}P_{2k})$.
We define graphs $G_0,G_1,G_2\dots$ recursively  as follows.
We start with $G_0:=G$.
Given~$G_{i}$, for every $v\in[n]$ we pick $v^+\in[n]\setminus\{v\}$ such that $vv^+$is an edge in $G_i$ and
\begin{itemize}
    \item if $i$ is odd, $v^+>v$ and $v^+$ is as large as possible, 
    \item if $i$ is even, $v^+<v$ and $v^+$ is as small as possible.
\end{itemize}
Note that $v^+$ may not exist e.g., if $i$ is odd and $v$ has no neighbour on its right.
We let $G_{i+1}$ be the graph obtained by deleting the edges $vv^+$ from $G_i$ for every $v\in[n]$. 

Suppose that there exists an edge $x_0y_0$ in $G_{2k-2}$.
By construction, it follows that there exists $x_1\in[y_0+1,n]$ such that $x_0x_1$ is an edge in~$G_{2k-3}$.
Again, this implies there exists $x_2\in[1,x_0-1]$ such that $x_2x_1$ is an edge in~$G_{2k-4}$.
Repeatedly applying this argument yields vertices $x_1,\dots,x_{2k-2}$ where $\dots<x_4<x_2<x_0<y_0<x_1<x_3<x_5<\dots$ and $x_ix_{i+1}$ is an edge in $G$ for every $i\in[2k-2]$.
But then $y_0x_0x_1\dots x_{2k-2}$ is an alternating $2k$-path in $G$, a contradiction.

We conclude that the graph~$G_{2k-2}$ must not have any edge.
Now, when we obtain $G_{i+1}$ from $G_i$, we delete at most one edge per vertex, namely the edge $vv^+$ for vertex $v$.
We do not delete any edge when considering the largest non-isolated vertex (for $i$ odd) or the smallest non-isolated vertex (for $i$ even).
Furthermore, it is easy to see that $G_i$ has at least $i$ isolated vertices (namely, $\{1,n,2,n-1,\dots\}$).
We conclude that $e(G_{i+1})\ge e(G_{i}) - (n - i - 1)$ for every $i$.
Hence, we have $0=e(G_{2k-2})\ge e(G_0) -2n(k-1)-(k-1)(2k-1)$ and so $e(G)\le 2(k-1)n-(k-1)(2k-1)$.
This implies $ex_<(n,\text{Alt-}P_{2k})\le 2(k-1)n -(k-1)(2k-1)$, as required.
\end{proof}

\section{Applications to Ramsey-type problems}\label{sec:appRamsey}

\subsection{Alternating paths}

Balko et. al (see Proposition 3 in \cite{bckk}) proved the following, where $R_<(\text{Alt-}P_t)$ is the smallest number of vertices of an ordered complete graph such that any $2$-edge-colouring yields a monochromatic copy of $\text{Alt-}P_t$.
$$R_<(\text{Alt-}P_t)\le2t-3+\sqrt{2t^2-8t+11}=(2+\sqrt2)t+o(t)\approx3.4t.$$
We can improve this upper bound using Theorem~\ref{thm:altpath}.

\begin{theorem}\label{thm:RamseyAltPaths}
For every $t\in\mathbb N$, we have
$$R_<(\text{Alt-}P_{t})\le3t+3.$$    
\end{theorem}

\begin{proof}
    Note that it suffices to show that $R_<(\text{Alt-}P_{t})\le3t$ if $t$ is even.
    Indeed, if $t$ is odd, it follows that $R_<(\text{Alt-}P_{t})\le R_<(\text{Alt-}P_{t+1})\le3(t+1)=3t+3$.

    Let $t=2k$ for some $k\in\mathbb N$ and let $m:=R_<(\text{Alt-}P_{2k})-1$.
    By definition of $m$, there exists a red/blue edge-coloured complete graph $K_m$ with vertex set $[m]$ which does not contain a monochromatic copy of $\text{Alt-}P_{2k}$.
    Note that $R_<(\text{Alt-}P_{2k})\ge|V(\text{Alt-}P_{2k})|=2k$ and so $m\ge2k-1$.

    Let $A:=\{xy\in E(K_m):x+y\le 2k-2\}$ and $B:=\{xy\in E(K_m):x+y\ge 2m-2k+4\}$.
    Observe that 
    \begin{align*}
    |A|=|\{xy\in E(K_m):x+y\le 2k-2\}| &=(2k-4)+(2k-6)+\dots+2\\
    & =2[(k-2)+(k-3)+\dots+1]=2\binom{k-1}{2}
    \end{align*}
    and similarly
    $$|B|=|\{xy\in E(K_m):x+y\ge 2m-2k+4\}|=2\binom{k-1}{2}.$$
    Furthermore, $A\cap B=\emptyset$.
    Indeed, we cannot have $2m-2k+4\le x+y\le 2k-2$ since $m\ge2k-1$.

    It follows that
    $$|E(K_m)\setminus(A\cup B)|=\binom{m}{2}-|A|-|B|=\binom{m}{2}-4\binom{k-1}{2}.$$
    By the pigeonhole principle, at least half of the edges in $E(K_m)\setminus(A\cup B)$ have the same color, say red.

    The key observation is that edges in $A\cup B$ can be arbitrarily recolored without creating a monochromatic copy of $\text{Alt-}P_{2k}$.
    Indeed, suppose there exists a copy of $\text{Alt-}P_{2k}$ in the complete uncolored graph with vertex set $[m]$, which contains some edge $xy\in A$ with $x<y$.
    Since a copy of $\text{Alt-}P_{2k}$ can be decomposed into two edge-disjoint nested matchings of $k$ and $k-1$ edges respectively, we conclude that there exists a nested $(k-1)$-matching $M$ using the edge $xy$.
    Now, every edge of $M$, which lies above $xy$ must intersect $[x-1]$, while $xy$ and every edge of $M$ below $xy$ must lie in $[x,y]$.
    Therefore, we have
    $$e(M)\le|[x-1]|+\frac{|[x,y]|}{2}\le(x-1)+\frac{y-x+1}{2}=\frac{y+x-1}{2}<k-1$$
    where the last inequality follows from the assumption that $xy\in A$ and so $x+y\le 2k-2$.
    Hence $M$ has strictly less than $k-1$ edges, a contradiction.
    By symmetry, an analogous argument works for $xy\in B$.

    We conclude that we can recolor the edges in $A\cup B$ with red without creating a monochromatic alternating $2k$-path.
    Let $G$ be the graph with vertex set $[m]$ whose edges are the red edges in $K_m$.
    Hence, we have
    $$e(G)=|A|+|B|+\frac{1}{2}\left(\binom{m}{2}-|A|-|B|\right)=\frac{1}{2}\left(\binom{m}{2}+4\binom{k-1}{2}\right).$$
    Note that $G$ does not contain a copy of an alternating $2k$-path, as that would be a red monochromatic copy in $K_m$.
    Therefore, we have
    $$e(G)\le ex_<(n,\text{Alt-}P_{2k})=2(k-1)m-(k-1)(2k-1).$$

    Hence 
    $$\frac{1}{2}\left(\binom{m}{2}+4\binom{k-1}{2}\right)\le 2(k-1)m-(k-1)(2k-1)$$
    which implies $m\le6k=3t$.
    That is, $R_<(\text{Alt-}P_{t})\le3t$ for $t$ even, as required.
\end{proof}

\begin{remark}
We learnt of simultaneous and independent work by Gaurav Kucheriya, Allan Lo, Jan Petr, Amedeo Sgueglia and Jun Yan on the Ramsey problem for the alternating path.
They obtained an even stronger upper bound than Theorem~\ref{thm:RamseyAltPaths}, replacing the~$3$ in the leading term with $2+\sqrt{2}/2$.
\end{remark}

\subsection{Non-nested matchings}

In~\cite{bgyt}, we studied the 2-coloring of the edges of the complete ordered graph on $[n]$. We asked the minimum number $n$ such that in any 2-coloring of the edges there is a monochromatic non-nested $k$-matching.
The current best bounds for this Ramsey question are as follows: 
$$3k-1\le R_<(\neg\text{Nest-}\mathcal M_k)\le4k-6,$$
where the lower bound is conjectured to be tight~\cite{bgyt} and the upper bound is unpublished~\cite{et25}.
The next lemma asserts that if the Tur{\'a}n question for non-nested matching behaves as Conjecture~\ref{conj:nn}, then one obtains a better upper bound for the analogous Ramsey question.

\begin{lemma}\label{lemma:nnestedRamsey}
    Suppose that $ex_<(n,\neg\text{Nest-}\mathcal M_k)=(k-1)n$ for every $n,k\in\mathbb N$ with $n\ge2k$.
    Then $R_<(\neg\text{Nest-}\mathcal M_k)\le(2+\sqrt{3})k$.
\end{lemma}

\begin{proof}
    Let $n:=R_<(\neg\text{Nest-}\mathcal M_k)-1$, hence there exists a red/blue colouring of the complete graph $K_n$ on vertex set $[n]$ such that there is no monochromatic non-nested $k$-matching.

    Observe that 
    $$|\{xy\in E(K_n):|x-y|\ge n-k+1 \}|=1+2+\dots+(k-1)=\binom{k}{2}$$ 
    and
    $$|\{xy\in E(K_n):|x-y|< n-k+1 \}|=\binom{n}{2}-\binom{k}{2}.$$
    By the pigeonhole principle, we may assume that there are at least $\frac{1}{2}(\binom{n}{2}-\binom{k}{2})$ red edges $xy$ such that $|x-y|<n-k+1$.
    Now, observe that a non-nested $k$-matching in $K_n$ cannot use any edge $xy$ such that $|x-y|\ge n-k+1$.
    Indeed, if $x<y$, then the remaining $k-1$ edges of the matching would be incident to vertices $v$ with $v<x$ or $v>y$, but there are only $k-2$ such vertices.

    We conclude that we can recolor the edges $xy$ satisfying $|x-y|\ge n-k+1$ with red without creating a monochromatic red non-nested $k$-matching.
    Let $G$ be the graph with vertex set $[n]$ whose edges are precisely the red edges in $K_n$.
    We have
    $$e(G)\le\frac{1}{2}\left(\binom{n}{2}-\binom{k}{2}\right)+\binom{k}{2}=\frac{1}{2}\left(\binom{n}{2}+\binom{k}{2}\right).$$
    In particular, $G$ does not contain a non-nested $k$-matching, and thus $e(G)\le ex_<(n,\neg\text{Nest-}\mathcal M_k)=(k-1)n$.
    It follows that
    $$\frac{1}{2}\left(\binom{n}{2}+\binom{k}{2}\right)\le n(k-1)$$
    and so $n^2-(4k-3)n+k^2-k\le0$.
    Since $3n\ge k$, we have $n^2-4kn+k^2\le0$, which implies $n\le(2+\sqrt{3})k$ as required.
\end{proof}

\section{Discussion}\label{sec:conclusion}

We conclude with some final open questions and remarks.
The most natural open problem is to fully settle the Tur\'an question for non-nested $k$-matchings.
Recall our main conjecture.

\nnconj*

Confirming Conjecture~\ref{conj:nn} would immediately yield better upper bounds for the analogous Ramsey question for non-nested matchings, as showed in Lemma~\ref{lemma:nnestedRamsey}.
In fact, it is easy to see from the proof of Lemma~\ref{lemma:nnestedRamsey} that replacing the $\binom{k-1}{2}$ term in the upper bound of Theorem~\ref{lemma:nnestedRamsey} with a sufficiently better constant would also yield non-trivial improvements for the Ramsey question.
In the following subsections, we list some additional open problems and observations.

\subsection{Tur\'an numbers of two or three forbidden patterns}

Recall Theorem~\ref{thm:cross,sep} determines the exact maximum number of edges of an ordered $n$-vertex graph containing no crossing or separated $k$-matching, provided $n\ge 2k^2$.
It is natural to consider the analogous question for other ordered matchings.

For $n\ge2k$, we have
$$ex_<(n,\{\text{Nest-}\mathcal M_k,\text{Sep-}\mathcal M_k\})=ex_<(n,\text{Nest-}\mathcal M_k)=2(k-1)n-(k-1)(2k-1).$$
To show this, it suffices to exhibit a graph with no nested or separated $k$-matching, and with $2(k-1)n-(k-1)(2k-1)$ edges. 
Consider the ordered graph $G$ with vertex set $[n]$ whose edge set consists of all edges incident to $[k-1]\cup[n-k+2,n]$.
It is easy to prove that $G$ satisfies the required properties.

On the other hand, the Tur\'an number $ex_<(n,\{\text{Nest-}\mathcal M_k,\text{Cross-}\mathcal M_k\})$ seems harder to determine.
A trivial lower bound is $ex_<(n,\{\text{Nest-}\mathcal M_k,\text{Cross-}\mathcal M_k\})\ge ex_<(n,\neg\text{Sep-}\mathcal M_k)$, the exact value of the right hand-side is given by Theorem~\ref{thm:nonsep}.
It is possible this is tight.

\begin{question}
Is it the case that $ex_<(n,\{\text{Nest-}\mathcal M_k,\text{Cross-}\mathcal M_k\})=ex_<(n,\neg\text{Sep-}\mathcal M_k)$? 
\end{question}

Another interesting question is to determine the maximum number of edges when forbidding a nested, crossing and separated $k$-matchings.
For $n\ge2k$ and $k\ge3$, we have 
$$(k-1)n\le ex_<(n,\{\text{Nest-}\mathcal M_k,\text{Cross-}\mathcal M_k,\text{Sep-}\mathcal M_k\})\le(k-1)n+O(k^3).$$
The upper bound follows immediately from Theorem~\ref{thm:cross,sep}, whereas the lower bound follows from the following construction.
Let $G$ be an ordered graph with vertex set $[n]$.
Let the edge set of $G$ consists of the union of all edges incident to the vertices $\{1\}\cup[k+1,2k-2]$ and the edges lying in $[2k-2]\cup\{n\}$.
Then $G$ contains no separated, nested or crossing $k$-matching (for $k\ge3$) and $G$ has precisely $(k-1)n$ edges.
This could be tight.

\begin{question}
Do we have $ex_<(n,\{\text{Nest-}\mathcal M_k,\text{Cross-}\mathcal M_k,\text{Sep-}\mathcal M_k\})=(k-1)n$ for $k\ge3$? 
\end{question}

\subsection{A counter-intuitive behaviour of the interval chromatic number}

Theorem~\ref{thm:PachTardos} by Pach and Tardos states that the interval chromatic number~$\chi_<(H)$ governs the asymptotic behaviour of~$ex_<(n,H)$.
In particular, if $\chi_<(H_1)<\chi_<(H_2)$ then (asymptotically) we have $ex_<(n,H_1)>ex_<(n,H_2)$.
Recall that we showed the following in Theorem~\ref{thm:cross,sep}:
$$ex_<(n,\text{Cross-}\mathcal M_k,\text{Sep-}\mathcal M_k)=(k-1)n+\Theta(k).$$
Here, $\chi_<(\text{Cross-}\mathcal M_k)=2$ and $\chi_<(\text{Sep-}\mathcal M_k)=k+1$.
Now, let $\mathcal M^*_k$ denote the disjoint union of two crossing matchings (possibly empty), say $\mathcal M_1$ and $\mathcal M_2$, such that $|\mathcal M_1|+|\mathcal M_2|=k$ and such that every pair of edges $e_1\in\mathcal M_1$ and $e_2\in\mathcal M_2$ are separated.
Formally, $\mathcal M^*_k$ is a family of ordered graphs, all of which have interval chromatic number $3$, with the exception of the crossing $k$-matching, which has interval chromatic number $2$.

In light of the previous consideration on the interval chromatic number, one may think that $ex_<(n,\mathcal M^*_k)$ should be (asymptotically) at most $ex_<(n,\text{Cross-}\mathcal M_k,\text{Sep-}\mathcal M_k)=(k-1)n+\Theta(k)$.
Indeed, while in both cases we are forbidding a crossing $k$-matching, in the former we forbid many ordered graphs with low interval chromatic number, whereas in the latter we forbid only one additional ordered graph with arbitrary large interval chromatic number.
Perhaps surprisingly, this turns out to be false: we have $ex_<(n,\mathcal M^*_k)\ge\left(k-2^{-(k-2)}\right)n + \Theta(k)$.

To prove this, it suffices to exhibit an extremal construction.
Let $G$ be a graph with vertex set~$[n]$ and $E(G)=E_1\cup E_2$ where $E_1$ is the set of edges incident to $\{1,2,\dots,k-2\}$ and
$$E_2:=\left\{xy\in\binom{[n]}{2}:|x-y|=2^i<k/2 \text{ and } x=t2^i+1 \text{ for some integers $i,t\ge 0$}\right\}.$$
Firstly, we have 
\begin{align*}
e(G)=|E_1\cup E_2|&=|E_1|+|E_2|-|E_1\cap E_2|\\
&=(k-2)n +(1+1/2+\dots+1/2^{\lfloor(k-1)/2\rfloor})n+\Theta(k) \\
&\ge\left(k-2^{-(k-2)}\right)n + \Theta(k).
\end{align*}
Finally, suppose for a contradiction that $G$ contains a copy of $\mathcal M^*_k$, say with $|\mathcal M_1|\ge |\mathcal M_2|$.
All edges $xy$ in $E_2$ satisfy $|x-y|<k/2\le|\mathcal M_1|$, hence $\mathcal M_1$ does not contain any edge in $E_2$ (since such edges cannot be used in a large crossing matching).
Also, no two edges in $E_2$ are crossing, hence $\mathcal M_2$ has at most one edge in $E_2$.
It follows that at least $k-1$ edges of $\mathcal M$ must lie in $E_1$, but this is not possible since all edges in $E_1$ are incident to $\{1,2,\dots,k-2\}$.

\subsection{`Crossing' and `nested' are not quite interchangeable}

In~\cite{kathe}, it was proved that there exists a bijection between ordered $n$-vertex graphs on $m$ edges which (A) do not contain a crossing $k$-matching and (B) do not contain a nested $k$-matching.      
In particular, this implies $ex_<(n,\text{Cross-}\mathcal M_k)=ex_<(n,\text{Nest-}\mathcal M_k)$.

The intuition that the terms `crossing' and `nested' are interchangeable is consistent with Kupitz's result that $ex_<(n,\neg\text{Cross-}\mathcal M_k)=(k-1)n$ and our conjecture that $ex_<(n,\neg\text{Nest-}\mathcal M_k)=(k-1)n$.
However, the following is an example, where interchanging the two terms produces different results.

Let $\mathcal M^{**}_k$ denote the disjoint union of two nested matchings (possibly empty), say $\mathcal M_1$ and $\mathcal M_2$, such that $|\mathcal M_1|+|\mathcal M_2|=k$ and such that every pair of edges $e_1\in\mathcal M_1$ and $e_2\in\mathcal M_2$ are separated.
Following Kupitz's proof of $ex_<(n,\neg\text{Cross-}\mathcal M_k)=(k-1)n$, one can show that $ex_<(n,\mathcal M^{**}_k)=(k-1)n$. 

If we replace `nested' with `crossing' in the definition of $\mathcal M^{**}_k$, we recover the definition of $\mathcal M^*_k$ from the previous subsection.
However, we showed that $ex_<(n,\mathcal M^*_k)\ge\left(k-2^{-(k-2)}\right)n + \Theta(k)$, which is significantly larger than $(k-1)n$.


\appendix

\section{Appendix}

Here, we show a self-contained proof of Theorem~\ref{dwsepar}.

\begin{theorem} \label{thm:sep}
For every $n,k\in\mathbb N$ with $n\ge 2k$ we have
$$ex_<(n,\text{Sep-}\mathcal M_k)
= T(n+1,k)-k+1.$$      
\end{theorem}

\begin{proof}
Let~$G$ be an arbitrary graph with vertex set $[n]$ which does not contain a copy of $\text{Sep-}\mathcal M_k$ and with $e(G)=ex_<(n,\text{Sep-}\mathcal M_k)$.
Let $\{x_1y_1,x_2y_2,\dots,x_t y_t\}$ be a separated $t$-matching in~$G$ where $x_1<y_1<x_2<y_2<\dots<x_t<y_t$.
We pick such matching so that~$t$ is as large as possible and, subject to this, the sum $y_1+\dots+y_t$ is as small as possible.

By assumptions, we have~$t\le k-1$.
For every integer $t<i\le k$, we set $y_i:=n+1$ and $y_0:=0$.
Suppose for some integer $0\le i\le k-1$, there is an edge~$e$ of~$G$ lying in $[y_i+1,y_{i+1}-1]$.
If $i<t$, then we can replace the edge~$x_{i+1}y_{i+1}$ with~$e$ to obtain a new separated $\ell$-matching with a smaller $y_1+\dots+y_t$, a contradiction.
If $i\ge t$, then we can add the edge~$e$ to the matching to obtain a larger separated matching, again a contradiction.
Therefore, for every integer $0\le i\le k-1$, there is no edge of~$G$ in $[y_i+1,y_{i+1}-1]$.

In particular, the underlying unordered graph of $G\setminus\{y_1,\dots,y_t\}$ is $k$-partite, meaning that $e(G\setminus\{y_1,\dots,y_t\})\le e(T(n-k+1,k))$ and so $e(G)\le e(T(n-k+1,k)) + (k-1)(n-1) -\binom{k}{2}$.
It is easy to check that the right hand-side of the last inequality is precisely $T(n+1,k)-k+1$, as required.
\end{proof}


\begin{thebibliography}{99}

\bibitem{bckk} M. Balko, J. Cibulka, K. Král and J. Kyn\v cl.
Ramsey numbers of ordered graphs.
{\it The Electronic Journal of Combinatorics}
(2020). \url{https://doi.org/10.37236/7816}


\bibitem{bgyt}J. Barát, A. Gyárfás, G. Tóth.
Monochromatic spanning trees and matchings in ordered complete graphs.
{\it Journal of Graph Theory}  {\bf 105}(4), (2024), 523--541.

\bibitem{B73} 
C.  Berge. Graphs and hypergraphs, 
North-Holland Publishing Company Amsterdam, (1973), pp 528. 

\bibitem{BKV03} P. Brass, Gy. Károlyi, P. Valtr. 
A Turán-type extremal theory of convex geometric graphs. Discrete and Computational Geometry: The Goodman-Pollack Festschrift. Berlin, Heidelberg: Springer Berlin Heidelberg, 2003. 275--300.


\bibitem{CapoyleasP} V. Capoyleas and J. Pach. 
A Tur\'an-type theorem on chords of a convex polygon. {\it Journal of Combinatorial Theory, Series B} {\bf 56}(1), (1992), 9--15.

\bibitem{cfls}D. Conlon, J. Fox, C. Lee and B. Sudakov.
Ordered Ramsey Numbers.
{\it Journal of Combinatorial Theory Series B}
{\bf 122}, (2017), 353--383.

\bibitem{layout} V. Dujmovi\'c, D.R. Wood.
On Linear Layouts of Graphs.
{\it Discrete Mathematics and Theoretical Computer Science} {\bf 6}, (2004), 339--358.

\bibitem{fh}Z. Füredi and P. Hajnal. Davenport-Schinzel theory of matrices. 
{\it Discrete Mathematics} {\bf 103}(3), (1992) 233--251.

\bibitem{twist}{H. Harborth and I. Mengersen}.
{Drawings of the complete graph with maximum number of crossings.}
{\it Congressus Numerantium}
{\bf 88}, (1992),
{225--228}.

\bibitem{hr}L.S. Heath and A.L. Rosenberg. 
Laying out graphs using queues. 
{\it SIAM Journal on Computing}, {\bf 21}(5), (1992), 927--958. 

\bibitem{kathe}J. Katheder, M. Kaufmann, S. Pupyrev, T. Ueckerdt.
Transforming Stacks into Queues: Mixed and Separated Layouts of Graphs.
\url{https://arxiv.org/pdf/2409.17776}

\bibitem{et25}N. Frankl, B. Kovács, L. Ködmön, B. Szabó K. Zólomy.
{\it personal communication}

\bibitem{Kupitz} Y.S. Kupitz. 
On pairs of disjoint segments in convex position in the plane. {\it Annals of Discrete Mathematics} {\bf 20}, (1984), 203--208.

\bibitem{mt} A. Marcus and G. Tardos. 
Excluded permutation matrices and the Stanley-Wilf conjecture. {\it Journal of Combinatorial Theory, Series A} {\bf 107}(1), (2004), 153--160.

\bibitem{o}L.T. Ollmann. 
On the book thicknesses of various graphs, in Proceedings of the 4th Southeastern Conference on Combinatorics, Graph Theory and Computing, F. Hoffman, R.B. Levow, and R.S.D. Thomas, eds., {\it Congressus Numerantium} {\bf 8}, (1973), 459.

\bibitem{PT06} J. Pach, G. Tardos. 
Forbidden paths and cycles in ordered graphs and matrices. {\it Israel Journal of Mathematics} {\bf 155.1} (2006), 359--380.


\bibitem{tardos} G. Tardos.
Extremal theory of ordered graphs.
{\it Proceedings of the International Congress of Mathematicians} Rio de Janeiro, Volume~4, (2018), 3253--3262.
\end{thebibliography}
\end{document}